\documentclass[preprint,12pt]{elsarticle}

\usepackage{amssymb}
\usepackage{enumerate}
\usepackage{csquotes}
\usepackage{mathtools}
\usepackage{amssymb}
 \newtheorem{thm}{Theorem}[section]
 \newtheorem{theorem}{Theorem}[section]
 
 \newtheorem{cor}[thm]{Corollary}
 
\newproof{proof}{\textbf{Proof}}

 \newtheorem{proposition}[thm]{Proposition}
 
 \newtheorem{definition}[thm]{Definition}
 \newtheorem{rem}[thm]{Remark}
 
 \newtheorem{example}{Example}
 \numberwithin{equation}{section}
\usepackage{epsfig,subfigure}
\usepackage{epstopdf}
\usepackage{geometry}
\usepackage{graphicx}
\usepackage{graphics}
\usepackage{subfigure}
\usepackage{tikz}
\usepackage{hyperref}

\begin{document}

\begin{frontmatter}



\title{Some inequalities related to Heinz mean constant with Birkhoff orthogonality}


\author{Kallal Pal} 
\author{Sumit Chandok\corref{cor}}
\cortext[cor]{Corrosponding author.}\ead{sumit.chandok@thapar.edu}
\affiliation{organization={Department of Mathematics,\\ Thapar Institute of Engineering and Technology},
            city={Patiala},
            postcode={147004}, 
            state={Punjab},
            country={India}}
\begin{abstract}
Motivated by the work of  Baronti et al. [J. Math. Anal. Appl. 252(2000)
124-146], where they defined the supremum of an arithmetic mean of the side lengths of a triangle, summing antipodal points on the unit sphere, we introduce a new geometric constant for Banach spaces, utilizing the Heinz means that interpolate between the geometric and arithmetic means associated with Birkhoff orthogonality. We discuss the bounds in Banach spaces and find the values of constant in Hilbert spaces.  We obtain the characterization of uniformly non-square spaces. We investigate the correlation between our notion of the Heinz mean constant and other well-known terms, viz., the modulus of convexity, modulus of smoothness, and rectangular constant.
Furthermore, we also
 give a characterization of the Radon plane with an affine regular hexagonal unit sphere.
\end{abstract}




\end{frontmatter}


\section{Introduction}
The quantitative framework for analyzing the geometry of Banach spaces is provided by the study of geometric constants, which is essential for addressing numerous problems in functional analysis. In addition to describing the basic structure of normed spaces, these constants show complex relationships that are both mathematically beautiful and novel. Using specified constants to study abstract properties of Banach spaces has led to impressive characterizations of inner product spaces and many fascinating open topics in modern mathematics.

Banach spaces do not accept a single or canonical notion of orthogonality, in contrast to Euclidean geometry; instead, several ideas of orthogonality exist, none of which is necessarily equivalent. Birkhoff \cite{birkhoff1935orthogonality} was the first to introduce a fundamental notion of orthogonality in 1935, and James \cite{james1947orthogonality} suggested isosceles orthogonality in 1947. 
These developments have stimulated extensive research on the interaction between generalized orthogonality and geometric constants, emphasizing the complexity of Banach space theory and its importance in both pure and applied mathematics.

Clarkson introduced the modulus of convexity \cite{clarkson1936uniformly} which controls quantitatively how far the midpoint of two distant unit vectors is from the unit sphere. If the modulus of convexity is non-negative, then the Banach space is uniformly convex. Banas \cite{banas1986moduli} introduced the modulus of smoothness, which measures how rounded the unit ball is inward (quantifying uniform smoothness).
In 1990, Gao and Lau \cite{gao1990geometry} defined the James constant, or
the non-square constant, to characterize uniformly non-square. 
 Recently, Baronti and Papini \cite{papini2022parameters} discussed James constant under the restriction of Birkhoff orthogonality. Inspired by the excellent works mentioned above, Du and Li \cite{du2023some} introduced two new constants: the modulus of convexity and the modulus of smoothness related to Birkhoff orthogonality, respectively. The meanings of these moduli related to Birkhoff orthogonality: if we take two vectors over the unit sphere satisfying Birkhoff orthogonality, then these moduli measure "how far" the middle point of the segment joining them must be from the unit sphere.\\
If $y$ and $-y$ are antipodal points on the unit sphere $S_X$, then $\|x+y\|$, $\|x-y\|$ 
 and 2 can be viewed as the lengths of the sides of the triangle $T_{xy}$
 with the vertices $x$, $y$ and $-y$
 lying on $S_X$. From a geometric point of view,
Baronti et al. \cite{baronti2000triangles} introduced a new constant $A_2(X)$ which
 can be regarded as the supremum of the arithmetic mean of the lengths of the sides of $T_{xy}$.

Du et al. \cite{du2025some} introduced a new geometric constant $A_2(X,B)$, which is closely associated with the modulus of smoothness 
 related to Birkhoff orthogonality.

For positive numbers $a,b>0$ and a parameter $\nu\in[0,1]$, Heinz mean is defined by
\begin{align*}
{H}_{\nu}(a,b)=\frac{a^\nu b^{1-\nu}+a^{1-\nu}b^\nu}{2}
\end{align*}

The Heinz mean interpolates smoothly between the geometric and arithmetic means, is convex in the parameter, attains its minimum at $\nu=\frac{1}{2}$, and satisfies the symmetry relation $H_{\nu}(X)=H_{1-\nu}(X)$
. These elegant properties make the Heinz mean an effective analytical bridge between multiplicative and additive structures, allowing for refined inequalities and deeper insights into convexity and smoothness.

Recently, Dinarvan \cite{dinarvand2019heinz} introduced new geometric constants based on averaging the lengths of the sides of triangle $T_{xy}$ by considering the Heinz
 mean, which are more general than the above constants for Banach spaces.

The Heinz mean, characterized by its symmetry and capacity to interpolate between multiplicative and additive structures, provides a natural framework for capturing delicate geometric features. At the same time,  Birkhoff orthogonality, closely associated with the geometric structure of Banach spaces, has long been recognized as a significant generalization of Euclidean orthogonality. By combining these two powerful ideas, we construct a novel parameter that extends the existing framework of geometric constants and highlights new connections between orthogonality, convexity, and smoothness. 

Based on these results, we present a geometric constant for Banach spaces, defined through the interaction of the Heinz mean and Birkhoff orthogonality. 

 The article's structure is as follows: After the introductory Section 1, we move on to Section 2, where we present some fundamental definitions, notations, and results. Section 3 defines a new geometric constant for Banach spaces, utilizing the Heinz means with Birkhoff orthogonality. We discuss the bounds in Banach spaces and find the values of constants in Hilbert spaces. We obtain the characterization of uniformly non-square spaces. We examine the relationship between the new geometric constants and the modulus of smoothness, modulus of convexity and the rectangular constant under Birkhoff orthogonality. Additionally, we provide a characterization of the Radon plane with an affine regular hexagonal unit sphere.

\section{Preliminaries}
The following are some notations and definitions that will be utilized in the subsequent sections.\\
Let $X$ be a Banach space and $S_X$ be a unit sphere of a Banach space $X$.
\begin{definition}
Let $x,y$ be two elements in a Hilbert space. Then an element $x \in X$ is said to be orthogonal 
to $y \in X$ (denoted by $x \perp y$) if $\langle x,y \rangle = 0$.   
\end{definition}

\begin{definition}\label{k51}\cite{birkhoff1935orthogonality}
In a normed linear space $X$, a vector $x$ is said to be Birkhoff-James orthogonal (BJ-orthogonality) to a vector $y$ ($x \perp_B y$) if the inequality $\|x+\lambda y\| \geq \|x\|$ holds for any real number $\lambda$.
\end{definition}

Recall that the space $X$ is called \textit{uniformly non-square} (see \cite{james1964uniformly}) 
if there exists $\delta > 0$ such that either 
\[
\frac{\|x-y\|}{2} \leq 1 - \delta, 
\quad \text{or} \quad 
\frac{\|x+y\|}{2} \leq 1 - \delta.
\]

\medskip
The James constant $J(X)$ and the Schäffer constant $S(X)$ are defined in \cite{gao1990geometry} as follows:
\[
J(X) = \sup \{ \min\{\|x+y\|, \|x-y\|\} : x,y \in S_X \},
\]
\[
S(X) = \inf \{ \max\{\|x+y\|, \|x-y\|\} : x,y \in S_X \}.
\]

Recently, Papini and Baronti  \cite{papini2022parameters} introduced the following constant:
\[
J(X,B) = \sup \{ \min\{\|x+y\|, \|x-y\|\} : x,y \in S_X,\, x \perp_B y \}.
\]
\begin{theorem}\label{j22}\cite{papini2022parameters}
Let $X$ be a Banach space. Then $J(X,B)=2$
 if and only if $X$ is not uniformly non-square.
\end{theorem}

Moreover, the following moduli related to Birkhoff orthogonality are defined by Du et al. \cite{du2023some} as follows:
\[
\delta_B(X) = \inf \left\{ 1 - \frac{\|x+y\|}{2} : x,y \in S_X,\, x \perp_B y \right\},
\]
\[
\rho_B(X) = \sup \left\{ 1 - \frac{\|x+y\|}{2} : x,y \in S_X,\, x \perp_B y \right\}.
\]

Joly \cite{joly1969caracterisations} defined the rectangular constant  $\mu'(X,B)$  as 
\[
\mu'(X,B) = \sup\left\{\frac{2}{\|x+y\|} : x,y \in S_X,\, x \perp_B y \right\}.
\]

Baronti et al. \cite{baronti2000triangles} defined another constant $A_2(X)$ as
\begin{align*}
A_2(X)=\left\{\frac{\|x+y\|+\|x-y\|}{2}: x,y \in S_X\right\}. 
\end{align*}
Du et al. \cite{du2025some} generalized the above constant using Birkhoff orthogonality and defined geometric constant $A_2(X,B)$ as
\begin{align*}
A_2(X,B)=\left\{\frac{\|x+y\|+\|x-y\|}{2}: x,y \in S_X, x\perp_B y\right\}. 
\end{align*}

\section{Main Results}

 For a given Banach space $X$, we define the following notion of Heinz mean:
\begin{align}\label{aq1}
\nonumber H_{\nu}(X,B) &= \sup_{x,y \in S_X} M_{\nu}( \|x+y\|, \|x-y\|)\\
&= \sup\left\{\frac{\|x+y\|^{\nu}\|x-y\|^{1-\nu} + \|x+y\|^{1-\nu}\|x-y\|^{\nu}}{2}: x,y \in S_X, x\perp_B y\right\},
\end{align}
where $0 \leq \nu \leq 1$.

\noindent Here, it is easy to see that
\begin{align*}
\min\{\|x+y\|, \|x-y\|\} &\leq \sqrt{\|x+y\| \|x-y\|}\\
&\leq \frac{\|x+y\|^{\nu}\|x-y\|^{1-\nu} + \|x+y\|^{1-\nu}\|x-y\|^{\nu}}{2}\\
&\leq \frac{\|x+y\| + \|x-y\|}{2}.
\end{align*}
So, it follows that
\begin{equation}\label{eq1}
J(X,B) \leq H_{\nu}(X,B) \leq A_2(X,B).
\end{equation}
\begin{proposition}\label{j1}
For a Banach space $X$, and $0 \leq \nu \leq 1$, we have $1 \leq H_{\nu}(X) \leq 2$.
\end{proposition}
\begin{proof}
  For $x,y \in S_X$ and $x\perp_B y$, we have $\|x\pm y\|\geq \|x\|=1$. From \eqref{aq1}, we have 
 \begin{align*}
 H_{\nu}(X,B)&= \sup\left\{\frac{\|x+y\|^{\nu}\|x-y\|^{1-\nu} + \|x+y\|^{1-\nu}\|x-y\|^{\nu}}{2}: x,y \in S_X, x\perp_B y\right\}\\
 &\geq \sup\left\{\frac{\|x+y\|^{\nu}\|x-y\|^{1-\nu} + \|x+y\|^{1-\nu}\|x-y\|^{\nu}}{2}: x,y \in S_X, \|x\pm y\|\geq 1 \right\}\\
 &\geq 1.    
 \end{align*}
 Using triangle inequality for $x,y \in S_X$ and $x\perp_B y$, we have
  \begin{align*}
 \frac{\|x+y\|^{\nu}\|x-y\|^{1-\nu} + \|x+y\|^{1-\nu}\|x-y\|^{\nu}}{2}\leq 2.    
\end{align*}
Taking suprimum over $x,y \in S_X$ and $x\perp_B y$, we get  $H_{\nu}(X,B)\leq 2$.
\end{proof}

\begin{example}
Let $X=\mathbb{R}^2$ endowed with a uniform norm
\[
\|x\| = \|(x_1, x_2)\| =
\|(x_1, x_2)\|_\infty.
\]
 Let 
$x = (1,1)$ and  $y= (-1,1)$.
It is clear that $x,y \in S_X$. Actually, we also have $x \perp_B y$. It is easy to see that $x \perp_B y$. Now, $\|x+y\|=2$ and $\|x-y\|=2$.
Here, it is easy to see that for all $0\leq \nu\leq 1$, we have 
\begin{align*}
H_\nu(X,B)\geq\frac{\|x+y\|^{\nu}\|x-y\|^{1-\nu} + \|x+y\|^{1-\nu}\|x-y\|^{\nu}}{2}=\frac{2^\nu.2^{1-\nu}+2^\nu.2^{1-\nu}}{2}=2.
\end{align*}
Therefore, $H_\nu(X,B)=2$.
\end{example}

\begin{proposition}\label{j01}
For a Hilbert space $X$, we have $ H_{\nu}(X,B)=\sqrt{2}$, for all $0 \leq \nu \leq 1$.
\end{proposition}
\begin{proof}
Suppose that $x,y \in S_X$ and $x\perp_B y$. So, $\|x\pm y\|=\sqrt{2}$.
Consider,
\begin{align*}
\frac{\|x+y\|^{\nu}\|x-y\|^{1-\nu} + \|x+y\|^{1-\nu}\|x-y\|^{\nu}}{2}=\frac{2\sqrt{2}}{2}=\sqrt{2}.    
\end{align*}
Taking supremum over $x,y \in S_X$, we have $H_{\nu}(X,B)=\sqrt{2}$.
\end{proof}
 The following example tells us that this conclusion may not hold that $X$ must be a Hilbert space when $H_{\nu}(X,B) = \sqrt{2}$.
\begin{example}
 Let $X$ be the space $\mathbb{R}^2$ endowed with the norm
\[
\|x\| = \|(x_1, x_2)\| = \max \left\{ |x_1|, |x_2|, \; \tfrac{1}{\sqrt{2}} (|x_1| + |x_2|) \right\}.
\]
Taking 
\[
x = \left( \tfrac{1}{\sqrt{2}}, \tfrac{1}{\sqrt{2}} \right), 
\quad 
y = \left( -\tfrac{1}{\sqrt{2}}, \tfrac{1}{\sqrt{2}} \right).
\]
It is obvious that $x, y \in S_X$ with $x \perp_B y$. And we also have
\[
\|x+y\| = \|x-y\| = \sqrt{2}.
\]
From the definition of $H_{\nu}(X,B)$, we have
\[
H_{\nu}(X,B)\geq\frac{\|x+y\|^{\nu}\|x-y\|^{1-\nu} + \|x+y\|^{1-\nu}\|x-y\|^{\nu}}{2}=\sqrt{2}.
\]
Therefore, $H_{\nu}(X,B)=\sqrt{2}$.
\end{example}

\begin{proposition}\label{pq1}
 For a Banach space $X$, \[J(X,B) \leq H_{\nu}(X,B)\leq {2^{\nu-1}[J(X,B)]^{1-\nu}+2^{-\nu}[J(X,B)]^{\nu}}.\]   
\end{proposition}
\begin{proof}
From \eqref{eq1}, we have 
$J(X,B) \leq H_{\nu}(X,B)$.\\
Case 1.: If $
\min\{\|x+y\|,\|x-y\|\}=\|x-y\|$, we have 
\begin{align*}
\|x+y\|^\nu\|x-y\|^{1-\nu}&=\|x+y\|^\nu[\min\{\|x+y\|,\|x-y\|\}]^{1-\nu}\\
&\leq (\|x\|+\|y\|)^\nu[\min\{\|x+y\|,\|x-y\|\}]^{1-\nu}\\
&=2^\nu[\min\{\|x+y\|,\|x-y\|\}]^{1-\nu},
\end{align*}
and
\begin{align*}
\|x-y\|^\nu\|x+y\|^{1-\nu}&=[\min\{\|x+y\|,\|x-y\|\}]^{\nu}\|x+y\|^{1-\nu}\\
&\leq [\min\{\|x+y\|,\|x-y\|\}]^{\nu}(\|x\|+\|y\|)^{1-\nu}\\
&=2^{1-\nu}[\min\{\|x+y\|,\|x-y\|\}]^{\nu}.
\end{align*}
Case 2.: If $
\min\{\|x+y\|,\|x-y\|\}=\|x+y\|$, similarly we have
\[\|x+y\|^\nu\|x-y\|^{1-\nu}\leq 2^{1-\nu}[\min\{\|x+y\|,\|x-y\|\}]^{\nu},\] and
\[\|x-y\|^\nu\|x+y\|^{1-\nu}\leq 2^{\nu}[\min\{\|x+y\|,\|x-y\|\}]^{1-\nu}.\]

From both the cases, we have
\begin{align*}
&\frac{\|x+y\|^{\nu}\|x-y\|^{1-\nu} + \|x+y\|^{1-\nu}\|x-y\|^{\nu}}{2}\\
& \leq \frac{2^\nu[\min\{\|x+y\|,\|x-y\|\}]^{1-\nu}+2^{1-\nu}[\min\{\|x+y\|,\|x-y\|\}]^{\nu}}{2}.
\end{align*}
Taking the supremum over $x,y \in S_X$ with $x\perp_B y$, we have
\begin{align*}
H_{\nu}(X,B)\leq\frac{2^\nu[J(X,B)]^{1-\nu}+2^{1-\nu}[J(X,B)]^{\nu}}{2}\\ 
={2^{\nu-1}[J(X,B)]^{1-\nu}+2^{-\nu}[J(X,B)]^{\nu}}
\end{align*}
\end{proof}
\begin{theorem}\label{t11}
A Banach space $X$ with $H_\nu(X,B)< 2$,  is uniformly non-square.
\end{theorem}
\begin{proof}
On the contrary, we suppose that $X$ is not uniformly non-square.
Then there exist $x,y \in S_X$ such that
\[
\|x \pm y\| \ge 2-\varepsilon^2, ~\text{for}~ \varepsilon \in (0,1/2).
\]
On rewriting we obtain
\[
2-\varepsilon^2 \le \|x+y\|
\le \|x+\varepsilon y\| + \|(1-\varepsilon)y\|
= \|x+\varepsilon y\| + 1-\varepsilon.
\]
Hence
$\|x+\varepsilon y\| \ge 1+\varepsilon-\varepsilon^2 > 1$. Also, $\|x+\varepsilon y\|\leq 1+\varepsilon$.

Define $\varpi_2=\frac{\|x+\varepsilon y\|-1}{\varepsilon}x-y$, 
$\bar \varpi_2=\frac{\varpi_2}{\|\varpi_2\|}$,
and $\phi:[0,1]\to\mathbb{R}$ by $\phi(\alpha)=\|y+\alpha \varpi_2\|$.
It is easy to see that $\phi$ is convex and 
$\phi\!\left(\frac{1}{\|x+\varepsilon y\|}\right)
=\frac{\|x+\varepsilon y\|-1}{\varepsilon}
=\phi(1)$.
Thus, $\phi$ attains its minimum at some
$t\in\left[\frac{1}{\|x+\varepsilon y\|},\,1\right]
\subset\left[{1+\varepsilon},1\right]
\quad (\text{since } \|x+\varepsilon y\|\le 1+\varepsilon)$,
which implies that $y+t \varpi_2 \perp_B \varpi_2$.

Choose $\varpi_1=y+t \varpi_2$ and $\bar \varpi_1=\varpi_1/\|\varpi_1\|$.
Since Birkhoff orthogonality is homogeneous, $\bar \varpi_1\perp_B \bar \varpi_2$ with $\|\bar \varpi_1\|=\bar \varpi_2\|=1$.

By the convexity of the norm,
$1-\varepsilon^2 \le \|s x \pm (1-s)y\| \le 1$, for all  $s\in[0,1]$. It implies that
$\|\alpha x \pm \beta y\|
\in \bigl[(1-\varepsilon^2)(\alpha+\beta),\,\alpha+\beta\bigr]$
for all non-negative scalars $\alpha,\beta$.\\
Hence,
\begin{align*}
\|\varpi_2\|&
=\left\|\frac{\|x+\varepsilon y\|-1}{\varepsilon}x-y\right\|
\in \bigl[(1-\varepsilon^2)(2-\varepsilon),\,2\bigr],\\
\|\varpi_1\|&
=\left\|t\frac{\|x+\varepsilon y\|-1}{\varepsilon}x+(1-t)y\right\|
\in \bigl[(1-\varepsilon^2)(1-\varepsilon),\,1\bigr],\\
\left\|\varpi_1+\frac{\varpi_2}{2}\right\|&
=\left\|\left(t+\frac12\right)\frac{\|x+\varepsilon y\|-1}{\varepsilon}x
-\left(t-\frac12\right)y\right\|
\ge (1-\varepsilon^2)\left(\frac{2}{1+\varepsilon}-\frac{3\varepsilon}{2}\right),\\
&\hspace{-1cm}\text{and}~~
\left\|\varpi_1-\frac{\varpi_2}{2}\right\|\geq \|u\|\geq (1-\epsilon^2)(1-\epsilon).
\end{align*}
Therefore
\begin{align*}
\|\bar \varpi_1+\bar \varpi_2\|
&\ge \left\|\varpi_1+\frac{\varpi_2}{2}\right\|
 - \|\varpi_1-\bar \varpi_1\| - \left\|\bar \varpi_2-\frac{\varpi_1}{2}\right\| \\
&= \left\|\varpi_1+\frac{\varpi_2}{2}\right\| - (1-\|\varpi_1\|) - \frac{2-\|\varpi_2\|}{2} \\
&\ge \frac{2-2\varepsilon^2}{1+\varepsilon} - 3\varepsilon-2\epsilon^2\geq \frac{2-2\varepsilon^2}{1+\varepsilon} - 5\varepsilon\\
&=\frac{2-5\varepsilon-7\varepsilon^2}{1+\varepsilon}.
\end{align*}
and 
\begin{align*}
 \|\bar \varpi_1-\bar \varpi_2\| &\ge \left\|\varpi_1-\frac{\varpi_2}{2}\right\|
 - \|\varpi_1-\bar \varpi_1\| - \left\|\bar \varpi_2-\frac{\varpi_2}{2}\right\| \\  
&= \left\|\varpi_1-\frac{\varpi_2}{2}\right\| - (1-\|\varpi_1\|) - \frac{2-\|\varpi_2\|}{2} \\ 
&\geq 1-\frac{5\varepsilon}{2}-3\varepsilon^2+\varepsilon^3\\
&\geq 1-4\varepsilon.
\end{align*}
Using the definition, we get
\begin{align*}
H_\nu(X,B)&\geq\frac{\|\bar \varpi_1+\bar \varpi_2\|^{\nu}\|\bar \varpi_1-\bar \varpi_2\|^{1-\nu} + \|\bar \varpi_1+\bar \varpi_2\|^{1-\nu}\|\bar \varpi_1-\bar \varpi_2\|^{\nu}}{2}\\
&\geq \frac{(\frac{2-5\varepsilon-7\varepsilon^2}{1+\varepsilon})^\nu(1-4\epsilon)^{1-\nu}+(\frac{2-5\varepsilon-7\varepsilon^2}{1+\varepsilon})^{1-\nu}(1-4\epsilon)^\nu}{2}.
\end{align*}
Letting $\varepsilon\rightarrow 0$, we get $H_\nu(X,B)\geq 2$, a contradiction.
\end{proof}

\begin{theorem}\label{j3}
Let $X$ be a Banach space. Then $H_{\nu}(X,B)=2$
 if and only if $X$ is not uniformly non-square.
\end{theorem}
\begin{proof}
Let, $X$ be not uniformly non-square. Then $J(X,B)=2$. From Theorem \ref{j22}, it is easy to see that $H_\nu(X,B)=2$.\\
Conversely, let $H_\nu(X,B)=2$. From the inequality of Theorem \ref{j22}, we have $J(X,B)\leq2$ and 
\begin{align*}
 2&\leq\frac{2^\nu[J(X,B)]^{1-\nu}+2^{1-\nu}[J(X,B)]^{\nu}}{2}\\
 &\leq \frac{2\nu+(1-\nu)[J(X,B)]+2(1-\nu)+\nu[J(X,B)]}{2}\\
 &=\frac{2+J(X,B)}{2}.
\end{align*}
This implies that $J(X,B)\geq 2$. Hence $J(X,B)=2$.
\end{proof}
\begin{rem}
    According to the definition of uniform convexity, it is easy to see that uniform convexity implies uniform non-squareness. Therefore, $H_{\nu}(X,B)=2$ for all $\alpha>\beta>0$ implies that $X$ is not a uniformly convex Banach space.
\end{rem}

\begin{proposition}
For any $\epsilon>0$, there  exists a uniformly convex Banach space $X$ such that $H_{\nu}(X,B)>2-\epsilon$.
\end{proposition}
\begin{proof}
Fix $\varepsilon > 0$. Notice that $\lim\limits_{p \to \infty} 2^{1 - 1/p} = 2$. Then there exists $p > 1$ such that  
\[
2^{1 - 1/p} \geq 2 - \varepsilon.
\]
Now consider the space $X = \mathbb{R}^2$ equipped with the norm
\[
\|x\| = \|(x_1, x_2)\| = \bigl( |x_1|^p + |x_2|^p \bigr)^{1/p}, 
\]
where $1 < p < \infty$. Note that the space $X$ is uniformly convex. Let 
\[
x = \bigl(2^{-1/p}, \, 2^{-1/p}\bigr) 
\quad \text{and} \quad 
y = \bigl(-2^{-1/p}, \, 2^{-1/p}\bigr).
\]
So, $\|x+y\|=2^{1-\frac{1}{p}}$ and $\|x-y\|=2^{1-\frac{1}{p}}$.
Also, it is evident that $x,y \in S_X$ such that $x \perp_B y$. Thus, it follows from the definition of $H_{\nu}(X,B)$ that we have
\begin{align*}
H_{\nu}(X,B) \geq \frac{\|x+y\|^{\nu}\|x-y\|^{1-\nu}+\|x+y\|^{1-\nu}\|x-y\|^{\nu}}{2} = 2^{1 - 1/p} \geq 2 - \varepsilon.
\end{align*}

This concludes the proof.
\end{proof}

It is important to note that every uniformly non-square Banach space is super-reflexive (see \cite{pisier2016martingales}) and has the fixed point property (see \cite{garcia2006uniformly}). Thus, we conclude the following result. 
\begin{cor}
Suppose that $X$ is a Banach space satisfying $H_\nu(X,B)< 2$. Then $X$ is super-reflexive and has the fixed point property.
\end{cor}

\begin{theorem}
 For a Banach space $X$, $H_{\nu}(X,B)\leq (1-\delta_B(X))^{\nu}+(1-\delta_B(X))^{1-\nu}$.   
\end{theorem}
\begin{proof}
For each pair of elements $x, y \in S_X$ such that $x\perp_B y$,  from the definition of $\delta_B(X)$, we have
$\|x + y\|\leq 2(1-\delta_B(X))$. Since $x, y \in S_X$, $\|x-y\|\leq \|x\|+\|y\|=2$. 
Consider, 
\begin{align*}
 &\frac{\|x+y\|^{\nu}\|x-y\|^{1-\nu} + \|x+y\|^{1-\nu}\|x-y\|^{\nu}}{2}\\
 &\hspace{1.5cm}\leq \frac{2^{\nu}(1-\delta_B(X))^{\nu}2^{1-\nu}+2^{1-\nu}(1-\delta_B(X))^{1-\nu}2^{\nu}}{2}\\
 &\hspace{1.5cm}= \frac{2(1-\delta_B(X))^{\nu}+2(1-\delta_B(X))^{1-\nu}}{2}\\
 &\hspace{1.5cm}=(1-\delta_B(X))^{\nu}+(1-\delta_B(X))^{1-\nu}.
 \end{align*}
Taking the supremum over $x,y \in S_X$ with $x\perp_B y$, we have
\[H_{\nu}(X,B)\leq (1-\delta_B(X))^{\nu}+(1-\delta_B(X))^{1-\nu}.\]
\end{proof}
\begin{proposition}
 For a Banach space $X$, $H_{\nu}(X,B)\geq\sqrt{2(1-\rho_B(X))}$.   
\end{proposition}
\begin{proof}
For each pair of elements $x, y \in S_X$ such that $x\perp_B y$,  from the definition of $\rho_B(X)$, we have $\rho_B(X)\geq 1-\frac{\|x+y\|}{2}$ that is $\|x+y\|\geq 2(1-\rho_B(X))$. Also, $x,y\in S_X$ and $x\perp_B y$ so, in particular, $\|x-y\|\geq \|x\|=1$. \\
From the definition of $H_{\nu}(X,B)$, we have
\begin{align*}
H_{\nu}(X,B)&= \sup\left\{\frac{\|x+y\|^{\nu}\|x-y\|^{1-\nu} + \|x+y\|^{1-\nu}\|x-y\|^{\nu}}{2}: x,y \in S_X, x\perp_B y\right\}\\
 &\geq \sup\left\{\frac{\|x+y\|^{\nu}\|x-y\|^{1-\nu} + \|x+y\|^{1-\nu}\|x-y\|^{\nu}}{2}: x,y \in S_X, \|x-y\|\geq 1 \right\}\\
&\geq \frac{\|x+y\|^{\nu} + \|x+y\|^{1-\nu}}{2}\\
 &\geq \frac{(2(1-\rho_B(X)))^{\nu}+(2(1-\rho_B(X)))^{1-\nu}}{2}\\
 &\geq \sqrt{(2(1-\rho_B(X)))^{\nu}(2(1-\rho_B(X)))^{1-\nu}}\\
 &=\sqrt{2(1-\rho_B(X))}.
 \end{align*}
 Therefore, $H_{\nu}(X,B)\geq\sqrt{2(1-\rho_B(X))}$.
\end{proof}

\begin{proposition}
For a Banach space $X$, $ \mu'(X,B)[H_{\nu}(X,B)]^2\geq2$.
\end{proposition}
\begin{proof}
According to the definition of $\mu'(X,B)\geq\frac{2}{\|x+y\|}$ that is, $\|x+y\|\geq \frac{2}{\mu'(X,B)}$.
Also, $x,y\in S_X$ and $x\perp_B y$ so, $\|x-y\|\geq \|x\|=1$. 
\begin{align*}
H_{\nu}(X,B)&= \sup\left\{\frac{\|x+y\|^{\nu}\|x-y\|^{1-\nu} + \|x+y\|^{1-\nu}\|x-y\|^{\nu}}{2}: x,y \in S_X, x\perp_B y\right\}\\
 &\geq \sup\left\{\frac{\|x+y\|^{\nu}\|x-y\|^{1-\nu} + \|x+y\|^{1-\nu}\|x-y\|^{\nu}}{2}: x,y \in S_X, \|x-y\|\geq 1 \right\}\\
&\geq \frac{\|x+y\|^{\nu} + \|x+y\|^{1-\nu}}{2}\\
 &\geq \frac{\left( \frac{2}{\mu'(X,B)} \right)^{\nu} + \left( \frac{2}{\mu'(X,B)} \right)^{1-\nu}}{2}\\
 &\geq \sqrt{\left( \frac{2}{\mu'(X,B)} \right)^{\nu}\left( \frac{2}{\mu'(X,B)} \right)^{1-\nu}}\\
 &=\sqrt{\frac{2}{\mu'(X,B)}}.
 \end{align*}
Therefore, $ \mu'(X,B)[H_{\nu}(X,B)]^2\geq2$.
\end{proof}

In Hilbert spaces, orthogonality is symmetry; specifically, $x \perp y$ implies $y \perp x$., this symmetric property is not generally holds in Banach spaces with respect to Birkhoff orthogonality.  
James \cite{james1947inner} demonstrated that a Banach space $X$ with $\dim X \geq 3$ possesses a norm induced by an inner product if and only if Birkhoff orthogonality is symmetric. The condition $\dim X \geq 3$ in the aforementioned result is necessary. This result, however, does not hold true in the case of two-dimensional spaces; in fact, there are numerous examples of real normed spaces that are two-dimensional in which Birkhoff orthogonality shows symmetrical behavior. These spaces are termed Radon planes, named after Radon, who was the first to characterize them. 
\begin{definition}
A two-dimensional normed space in which Birkhoff orthogonal
ity is symmetric is called a Radon plane.
\end{definition}
Du et al. \cite{du2025some} considered the constant $A_2(X,B)$ in Radon
planes and showed that the following results are true for any Radon plane $X$:
\begin{enumerate}
    \item $A_2(X,B)\leq \frac{3}{2}$.
    \item $A_2(X,B)=\frac{3}{2}$
 if and only if its unit sphere 
 is an affine-regular hexagon.
\end{enumerate}
 We analyze the $H_\nu(X,B)$ constant in Radon planes and illustrate that the range of values decreases. Furthermore, the constants $H_\nu(X,B)$ of the Radon planes with unit spheres are derived.

Since $A_2(X,B)\leq \frac{3}{2}$ holds for any Radon plane, we can get the following result
easily by \eqref{eq1}.
\begin{proposition}\label{op1}
For a Radon plane $X$, we have $H_\nu(X,B)\leq \frac{3}{2}$.
\end{proposition}
Moreover, the following
example will indicate that the upper bound shown in the above result is also sharp.

\begin{example}
Let $X$ be a Radon plane $l_\infty - l_1$, that is, the space $\mathbb{R}^2$ with the norm
\[
\|(x_1, x_2)\| = \begin{cases}
\|(x_1, x_2)\|_\infty & (x_1 x_2 \geq 0), \\
\|(x_1, x_2)\|_1 & (x_1 x_2 \leq 0).
\end{cases}
\]
Then  $H_\nu(X,B)=\frac{3}{2}$.
\end{example}
\proof Take $x = (1,1)$, $y = (-\frac{1}{2}, \frac{1}{2})$, then $x, y \in S_X$. In addition, one can easily
obtain $x \perp_B y$, from the discussion below:
\begin{itemize}
  \item \textbf{Case 1:} $-2\leq \lambda\leq -2$.
  Then,
  \[
  \|x + \lambda y\| = \left\| \left( 1-\frac{\lambda}{2}, 1+\frac{\lambda}{2} \right) \right\| = \max\left\{\left|1-\frac{\lambda}{2}\right|, \left|1+\frac{\lambda}{2}\right|\right\}=\left|1+\frac{\lambda}{2}\right|\geq 1=\|x\|.
  \]
 \item \textbf{Case 2:} $\lambda\leq -2$ and $ \lambda\geq 2$.
  Then,
  \[
\|x + \lambda y\| = \left\| \left( 1-\frac{\lambda}{2}, 1+\frac{\lambda}{2} \right) \right\| = \left|1-\frac{\lambda}{2}\right|+ \left|1+\frac{\lambda}{2}\right|\geq 1=\|x\|.
  \]
\end{itemize}
Now, $\|x+ y\|=\frac{3}{2}$ and $\| x-y\|=\frac{3}{2}$.\\
From the Proposition \eqref{op1} and the definition of $H_\nu(X,B)$, we have
\[
\dfrac{3}{2}\geq H_\nu(X,B) \geq \frac{\|x+y\|^{\nu}\|x-y\|^{1-\nu} + \|x+y\|^{1-\nu}\|x-y\|^{\nu}}{2}=\dfrac{3}{2}.
\]
Therefore, $H_\nu(X,B)=\frac{3}{2}$.
This completes the proof.

\begin{theorem}
Suppose that $X$ is a Radon plane. Then  $H_\nu(X,B)=\frac{3}{2}$ if and only if $S_X$ is an affine regular hexagon.
\end{theorem}
\begin{proof}
Assume that  $H_\nu(X,B)=\frac{3}{2}$. Then, from \eqref{eq1}, we have $\frac{3}{2}=H_\nu(X,B)\leq A_2(X,B)$. Since $A_2(X,B)\leq \frac{3}{2}$ for Radon plane. Therefore, $A_2(X,B)=\frac{3}{2}$.
Hence $S_X$ is an affine regular hexagon.

Conversely, if $S_X$ is an affine regular hexagon, then we can assume that there exist
$ u,v \in S_X$ satisfy that $\pm u, \pm v, \pm (u + v)$ are the vertices of $S_X$, see the following figure\\
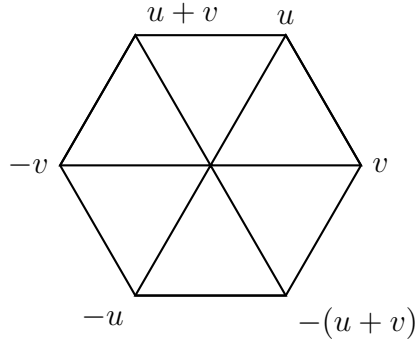
\begin{center}
\begin{tikzpicture}[scale=2, thick]
\coordinate (O) at (0,0);
\coordinate (A) at (60:1);    
\coordinate (B) at (0:1);     
\coordinate (C) at (-60:1);   
\coordinate (D) at (240:1);   
\coordinate (E) at (180:1);   
\coordinate (F) at (120:1);   

\draw (A)--(F)--(E)--(D)--(C)--(B)--cycle;
\draw (O)--(A)--(B)--(O);
\draw (O)--(C)--(D)--(O);
\draw (O)--(E)--(F)--(O);

\node[above] at (A) {$u$};
\node[above right] at (F) {$u+v$};
\node[right] at (B) {$v$};
\node[left] at (E) {$-v$};
\node[below left] at (D) {$-u$};
\node[below right] at (C) {$-(u+v)$};
\end{tikzpicture}
\end{center}
 Let $x =v$, $y = u+\dfrac{v}{2}$.Then it is evident that $x,y \in S_X$ and $x\perp_B y$.\\
 Now,
\[
\|x + y\| = \left\| v + \left( u + \tfrac{1}{2}v \right) \right\| 
= \left\| (u + v) + \tfrac{1}{2}v \right\| 
= \tfrac{3}{2} \left\| \tfrac{2}{3}(u + v) + \tfrac{1}{3}v \right\| 
= \tfrac{3}{2},
\]
and
\[
\|x - y\| = \left\| v - \left( u + \tfrac{1}{2}v \right) \right\| 
= \left\| -u + \tfrac{1}{2}v \right\| 
= \tfrac{3}{2} \left\| \tfrac{2}{3}(-u) + \tfrac{1}{3}v \right\| 
= \tfrac{3}{2}.
\]
\noindent Then it follows from the definition of $H_\nu(X,B)$ and the fact that 
$H_\nu(X,B)\leq \frac{3}{2}$ for any Radon plane that we derive
\[
\dfrac{3}{2}\geq H_\nu(X,B) \geq \frac{\|x+y\|^{\nu}\|x-y\|^{1-\nu} + \|x+y\|^{1-\nu}\|x-y\|^{\nu}}{2}=\dfrac{3}{2}.
\]
Therefore, $H_\nu(X,B)= \frac{3}{2}$.
This completes the proof.
\end{proof}





\end{document}